\documentclass{amsart}

\usepackage{amsmath,amssymb,amsthm,url,bm,microtype,parskip,enumitem}

\newtheorem{lemma}{Lemma}[section]

\newtheorem{cor}[lemma]{Corollary}
\newtheorem{thm}[lemma]{Theorem}

\newtheorem{thm?}[lemma]{Theorem?}

\newtheorem{ques}{Question}
\newtheorem{conj}[lemma]{Conjecture}

\theoremstyle{definition}
\newtheorem{remark}[lemma]{Remark}

\providecommand{\bysame}{\leavevmode\hbox to3em{\hrulefill}\thinspace}

\begin{document}
\title{The number of atoms in an atomic domain}

\author{Pete L. Clark}
\author{Saurabh Gosavi}
\author{Paul Pollack}

\email{pete@math.uga.edu}
\email{saurabh.gosavi@uga.edu}
\email{pollack@uga.edu}

\address{Department of Mathematics \\ Boyd Graduate Studies Research Center \\ University
of Georgia \\ Athens, GA 30602\\ USA}


%

\newcommand{\etalchar}[1]{$^{#1}$}
\newcommand{\F}{\mathbb{F}}
\newcommand{\et}{\textrm{\'et}}
\newcommand{\ra}{\ensuremath{\rightarrow}}
\newcommand{\FF}{\F}
\newcommand{\Z}{\mathbb{Z}}
\newcommand{\N}{\mathcal{N}}
\newcommand{\ch}{}
\newcommand{\R}{\mathbb{R}}
\newcommand{\PP}{\mathbb{P}}
\newcommand{\pp}{\mathfrak{m}}
\newcommand{\C}{\mathbb{C}}
\newcommand{\Q}{\mathbb{Q}}
\newcommand{\tpqr}{\widetilde{\triangle(p,q,r)}}
\newcommand{\ab}{\operatorname{ab}}
\newcommand{\Aut}{\operatorname{Aut}}
\newcommand{\gk}{\mathfrak{g}_K}
\newcommand{\gq}{\mathfrak{g}_{\Q}}
\newcommand{\OQ}{\overline{\Q}}
\newcommand{\Out}{\operatorname{Out}}
\newcommand{\End}{\operatorname{End}}
\newcommand{\Gon}{\operatorname{Gon}}
\newcommand{\Gal}{\operatorname{Gal}}
\newcommand{\CT}{(\mathcal{C},\mathcal{T})}
\newcommand{\ttop}{\operatorname{top}}
\newcommand{\lcm}{\operatorname{lcm}}
\newcommand{\Div}{\operatorname{Div}}
\newcommand{\OO}{\mathcal{O}}
\newcommand{\rank}{\operatorname{rank}}
\newcommand{\tors}{\operatorname{tors}}
\newcommand{\IM}{\operatorname{IM}}
\newcommand{\CM}{\operatorname{CM}}
\newcommand{\Frac}{\operatorname{Frac}}
\newcommand{\Pic}{\operatorname{Pic}}
\newcommand{\coker}{\operatorname{coker}}
\newcommand{\Cl}{\operatorname{Cl}}
\newcommand{\loc}{\operatorname{loc}}
\newcommand{\GL}{\operatorname{GL}}
\newcommand{\PSL}{\operatorname{PSL}}
\newcommand{\Frob}{\operatorname{Frob}}
\newcommand{\Hom}{\operatorname{Hom}}
\newcommand{\Coker}{\operatorname{\coker}}
\newcommand{\Ker}{\ker}
\renewcommand{\gg}{\mathfrak{g}}
\newcommand{\sep}{\operatorname{sep}}
\newcommand{\new}{\operatorname{new}}
\newcommand{\Ok}{\mathcal{O}_K}
\newcommand{\ord}{\operatorname{ord}}
\newcommand{\MM}{\mathcal{M}}
\newcommand{\mm}{\mathfrak{m}}
\newcommand{\Ohell}{\OO_{p^{\infty}}}
\newcommand{\ff}{\mathfrak{f}}
\renewcommand{\N}{\mathbb{N}}
\newcommand{\Gm}{\mathbb{G}_m}
\newcommand{\Spec}{\operatorname{Spec}}
\newcommand{\MaxSpec}{\operatorname{MaxSpec}}
\newcommand{\qq}{\mathfrak{q}}
\newcommand{\cof}{\operatorname{cf}}
\newcommand{\lub}{\operatorname{lub}}
\newcommand{\glb}{\operatorname{glb}}
\newcommand{\redbox}{{\color{red}{\blacksquare}}}
\newcommand{\CKn}{\text{CK}(n)}
\newcommand{\CKnm}{\text{CK}(n,m)}
\newcommand{\CK}{{\rm CK}}

\setlength{\parskip}{6pt}
\setlength{\parindent}{0pt}
\newcommand{\forceindent}{\leavevmode{\parindent=15pt\indent}}
\setlength{\topsep}{0.35cm plus 0.2cm minus 0.2cm}

\makeatletter
\def\thm@space@setup{%
  \thm@preskip=0.25cm plus 0.3cm minus 0.16cm
  \thm@postskip=0.01cm plus 0cm minus 0.16cm
}
\makeatother

\begin{abstract}
We study the number of atoms and maximal ideals in an atomic domain with finitely many atoms and no prime elements.  We show in particular that for all $m,n \in \Z^+$ with $n \geq 3$ and $4 \leq m \leq \frac{n}{3}$ there is an atomic domain with precisely $n$ atoms, precisely $m$ maximal ideals and no prime elements.  The proofs use both commutative algebra and additive number theory.
\end{abstract}

\maketitle


\section{Introduction}
\subsection{Terminology}
Let $B \subset\Z^+$ be infinite, and let $A \subset B$.  We say \textbf{the relative density of $\bm A$ in $\bm B$ is} \boldsymbol{$\delta$} if
\[ \lim_{n \ra \infty} \frac{ \# (A \cap [1,n])}{\# (B \cap [1,n])} = \delta. \]
We say \textbf{the density of $\bm A$ is} \boldsymbol{$\delta$} if the relative density of $A$ in $\Z^+$ is $\delta$.
Let $R$ be a domain (a commutative ring without zero-divisors).  Let $R^{\bullet} = R \setminus \{0\}$, $R^{\times}$ be the unit group and $R^{\circ} = R^{\bullet} \setminus R^{\times}$.  An element $p \in R^{\circ}$ is \textbf{prime} if the ideal $(p)$ is prime; $p \in R^{\circ}$ is an \textbf{atom} if for all $x,y \in R$, $p = xy \implies x \in R^{\times}$ or $y \in R^{\times}$.  A domain $R$ is \textbf{atomic} if every $x \in R^{\circ}$ is a finite product of atoms and \textbf{factorial} if every $x \in R^{\circ}$ is a finite product of primes.  A domain is factorial iff it is atomic and all atoms are prime \cite[Thm. 15.8]{Clark-CA}.  A domain is \textbf{primefree} if it has no prime elements.  For a domain $R$ and a cardinal $\kappa$, ``$R$ has $\kappa$ atoms'' means there is a set $\mathcal{P}$ of atoms of $R$ of cardinality $\kappa$ such that
every atom of $R$ is associate to a unique $p \in \mathcal{P}$.  ``$R$ has $\kappa$ maximal ideals'' means the set $\MaxSpec R$ of maximal ideals of $R$ has cardinality $\kappa$.
\\ \forceindent
A \textbf{Cohen-Kaplansky domain} is an atomic domain with $\kappa < \aleph_0$ atoms.  Let $m,n \in \Z^+$ and let $q$ be a prime power. A \textbf{CK$\bm{(n)}$-domain} is an atomic domain with $n$ atoms; a \textbf{CK$\bm{(n;q)}$}-domain is a CK$(n)$-domain which is also an $\F_q$-algebra; and a \textbf{CK$\bm{(n;0)}$-domain} is a $\CKn$-domain of characteristic $0$.   A
\textbf{CK$\bm{(n,m)}$-domain} is a CK$(n)$-domain with exactly $m$ maximal ideals; a \textbf{CK$\bm{(n,m;q)}$-domain} is a $\CKnm$-domain which is also an $\F_q$-algebra; and a \textbf{CK$\bm{(n,m;0)}$}-domain is a $\CKnm$-domain of characteristic $0$.

\subsection{Main Results}
For any cardinal $\kappa$ there is a domain with $\kappa$ atoms: if $\kappa$ is finite we may take a localization of $\Z$, and if $\kappa$ is infinite we may take $k[t]$ for a field $k$ of cardinality $\kappa$.  These examples are not so interesting: they are factorial domains, so every atom is prime.  Our point of departure in this note is the following:
\begin{ques}
\label{QUES1}
For which $\kappa$ is there a \emph{primefree} atomic domain with $\kappa$ atoms?
\end{ques}
In \cite{Cohen-Kaplansky46} Cohen-Kaplansky showed that a primefree atomic domain which is not a field has at least three atoms and that there are \emph{local} primefree atomic domains with $n$ atoms for $3 \leq n \leq 10$ \cite{Cohen-Kaplansky46}.  The case $\kappa \geq \aleph_0$ has already been handled.

\begin{thm}
\label{THM1.1}
Let $\kappa$ be an infinite cardinal.  Then:
\begin{enumerate}[noitemsep,topsep=0pt,label={\alph*}{\rm)},font=\rm]
\item \cite[Thm. 5.2]{Clark15} There is a primefree local atomic domain with $\kappa$ atoms.
\item \cite[Thm. 5.5]{Clark15} There is a primefree Dedekind domain with $\kappa$ atoms.
\end{enumerate}
\end{thm}
\noindent
We are left with the (more interesting) case of $\kappa < \aleph_0$ atoms.  Question \ref{QUES1} becomes: for which $n \geq 3$ is there a primefree $\CKn$-domain?  In this form it was raised by Coykendall-Spicer \cite{Coykendall-Spicer12}, who showed there are primefree $\CKn$-domains for all $n \geq 3$ conditionally on the conjecture that every even $n \geq 8$ is a sum of two \emph{distinct} primes.  They derived this as a consequence of the following result.

\begin{thm}[{Coykendall-Spicer \cite{Coykendall-Spicer12}}]
\label{MAINCSCOR}
\label{THM1.2}
For any primes $p_1 < \ldots < p_m$ there is a primefree \CK$(\sum_{j=1}^m (p_j+1),m;0)$-domain.
\end{thm}
\noindent
In the 1930's Chudakov, Estermann and van der Corput showed that the subset of even positive integers which are a sum of two primes has relative density $1$ \cite{Chudakov37,Chudakov38}, \cite{Estermann38}, \cite{vanderCorput37}.  From this
and Theorem \ref{MAINCSCOR} it follows that the set of $n \in \Z^+$ for which there is a primefree \CK$(n,2;0)$-domain or a primefree \CK$(n,3;0)$-domain has density $1$.
 We will give a stronger result with a similar proof, so we omit the details.  By making a different -- more elementary -- analytic argument, we may deduce from Theorem \ref{MAINCSCOR} an answer to Question \ref{QUES1}.

\begin{thm}
\label{MAINTHM}
\label{THM1.3}
For all $n \geq 3$ there is a primefree \CK$(n)$-domain.
\end{thm}
\noindent
Here is the key idea of the proof: whereas Coykendall-Spicer apply Theorem \ref{MAINCSCOR} with $m \in \{2,3\}$, to get primefree CK$(n,m)$-domains, we may choose $m$ in terms of $n$.  In fact it suffices to prove the following result.

\begin{thm}
\label{PAULTHM}
\label{THM1.4}
For all $n \geq 6$ there are primes $p_1 < \ldots < p_m$ such that $n = \sum_{j=1}^m (p_j + 1)$.
\end{thm}
\noindent
This maneuver leads us to ask a refined version of Question \ref{QUES1} in the finite case.

\begin{ques}
\label{QUES2}
For which $m,n \in \Z^+$ is there a primefree $\CK(n,m)$-domain?
\end{ques}
\noindent
The main goal of this note is to address Question \ref{QUES2}.  We give a complete answer for $m =4$, an answer up to finitely many $n$ for $m =3$, and an answer up to density $0$ for $m =2$.  Moreover we conjecture a complete answer for all $m \geq 2$.  For $m = 1$ we (only) record an implication of an old result of Cohen-Kaplansky.  In more detail:

\begin{thm}
\label{LITTLEMAIN}
\label{THM1.5}
If there is a primefree $\CK(n,m)$-domain, then
$m \leq \frac{n}{3}$.
\end{thm}

\begin{thm}
\label{MAINERTHM1}
\label{THM1.6}
Let $m$ and $n$ be positive integers.
\begin{enumerate}[noitemsep,topsep=0pt,label={\alph*}{\rm)},font=\rm]
\item If $n \geq 10$ is even and $m \in [3,\frac{n}{3}]$, there is a primefree \CK$(n,m;0)$-domain.
\item If $n \geq 13$ is odd and $m \in [4,\frac{n}{3}]$, there is a primefree \CK$(n,m;0)$-domain.
\item If $n$ is sufficiently large, there is a primefree \CK$(n,3;0)$-domain.
\end{enumerate}
\end{thm}

\begin{thm}\label{MAINERTHM2}\label{THM1.7}\mbox{ }
\begin{enumerate}[noitemsep,topsep=0pt,label={\alph*}{\rm)},font=\rm]
\item The set of $n \in \Z^+$ for which there is a primefree \CK$(n,2;0)$-domain has
    density $1$.
\item Conditionally on the Goldbach Conjecture -- Conjecture \ref{GOLDBACHCONJ} in $\S$3 -- for every even $n \geq 6$ there is a primefree \CK$(n,2;0)$-domain.\item Conditionally on Schinzel's generalization of the Goldbach Conjecture-- Conjecture \ref{SCHINZELCONJ} in $\S$3 -- for every sufficiently large odd integer $n$ there is a primefree \CK$(n,2;0)$-domain.
\end{enumerate}
\end{thm}

\begin{conj} Let $n$ be a positive integer.
\label{CONJ1}
\begin{enumerate}[noitemsep,topsep=0pt,label={\alph*}{\rm)},font=\rm]
\item There is a primefree \CK$(n,2;0)$-domain iff $n \geq 6$.
\item There is a primefree \CK$(n,3;0)$-domain iff $n \geq 9$.
\end{enumerate}
\end{conj}


\begin{thm}
\label{MAINERPAUL}
\label{THM1.8}
The set of primes $n$ such that there is a primefree \CK$(n,1)$-domain has density $0$ inside the set of all primes.
\end{thm}
\noindent
We will use results in the theory of Cohen-Kaplansky domains due to Cohen-Kaplansky \cite{Cohen-Kaplansky46} and Anderson-Mott \cite{Anderson-Mott92} as well as the following result.

\begin{thm}[Globalization Theorem]\mbox{ }
\label{SAURABHTHM}
\label{THM1.9}
\begin{enumerate}[noitemsep,topsep=0pt,label={\alph*}{\rm)},font=\rm]
\item Let $q$ be either $0$ or a prime power.  Let $M,n_1,\ldots,n_M,m_1,\ldots,m_M \in \Z^+$.  Suppose for all $1 \leq j \leq M$ there
is a primefree \CK$(n_j,m_j;q)$-domain.  Then there is a primefree \CK$(\sum_{j=1}^M n_j,\sum_{j=1}^M m_j;q)$-domain.
\item If $q_1,\ldots,q_m$ are prime powers and $d_1,\ldots,d_m \geq 2$ are integers, there is a primefree \CK$(\sum_{j=1}^m \frac{q_j^{d_j}-1}{q_j-1},m;0)$-domain.  If $q_1,\ldots,q_m$ are all powers of a prime power $q$, there is a primefree
\CK$(\sum_{j=1}^m \frac{q_j^{d_j}-1}{q_j-1},m;q)$-domain.
\end{enumerate}
\end{thm}


Here is another application of Theorem \ref{SAURABHTHM}.

\begin{thm}
\label{MAINQ}
\label{THM1.10}
Let $q$ be a prime power.
\begin{enumerate}[noitemsep,topsep=0pt,label={\alph*}{\rm)},font=\rm]
\item If $R$ is a primefree atomic domain which is an $\F_q$-algebra and not a field,  then $R$ has at least $q+1$ atoms.
\item  For all $n \geq 3$, there is a primefree \CK$(n;2)$-domain.
\item If $q$ is even, then for all $n \geq 2q^2-q$ there is a primefree \CK$(n;q)$-domain.
\item If $q$ is odd, then for all $n \geq 2q^2-q+1$ there is a primefree \CK$(n;q)$-domain.
\end{enumerate}
\end{thm}

\subsection{Structure of the Paper}
In \S2 we give material on Cohen-Kaplansky domains.  In $\S$3 we will prove Theorems \ref{THM1.3}--\ref{THM1.7} and \ref{THM1.8}--\ref{THM1.10} and give supporting arguments for Conjecture \ref{CONJ1}.  Final comments are given in $\S$4.

\section{Preliminaries on Cohen-Kaplansky Domains}
 There is a beautiful structure theory for Cohen-Kaplansky domains pioneered by Cohen-Kaplansky \cite{Cohen-Kaplansky46} and enhanced by Anderson-Mott \cite{Anderson-Mott92}.  In this section we recall some of these results -- and quick consequences of them -- for later use.

\begin{thm}[\cite{Cohen-Kaplansky46}]
\label{02.1}
A Cohen-Kaplansky domain is Noetherian and has only finitely many nonzero prime all ideals, all of which are maximal.
\end{thm}

\begin{thm}
\label{1.2}
\label{02.2}
Let $R$ be a semilocal domain, with maximal ideals $\pp_1,\ldots,\pp_m$.
\begin{enumerate}[noitemsep,topsep=0pt,label={\alph*}{\rm)},font=\rm]
\item $R$ is Cohen-Kaplansky iff $R_{\mm_j}$ is Cohen-Kaplansky for all $1 \leq j \leq m$.
\item Every atom $p$ of $R$ lies in $\pp_j$ for exactly one $j$, and $R \hookrightarrow R_{\mm_j}$ induces a bijection from the atoms of $R$ lying in $\mm_j$ to the atoms of $R_{\mm_j}$.
\item Suppose $R$ is Noetherian of dimension one.  The following are equivalent:
\begin{enumerate}[noitemsep,topsep=0pt,label={{\rm(}\roman*{\rm)}},font=\rm]
\item $R$ has a prime element.
\item For at least one $j$, $R_{\mm_j}$ has a prime element.
\item For at least one $j$, $R_{\mm_j}$ is a DVR.
\end{enumerate}
\end{enumerate}
\end{thm}
\begin{proof}
Parts a) and b) are results from \cite{Cohen-Kaplansky46}.  c) (i) $\implies$ (ii): if $p \in R$ is prime, then since $\dim R = 1$ we have $(p) = \pp_j$ for some $j$ and $p$ is a prime element of $R_j$.
(ii) $\iff$ (iii): for all $j$, $R_{\mm_j}$ is one-dimensional local Noetherian, hence is a DVR iff the maximal ideal is principal \cite[Thm. 17.19]{Clark-CA}.
(iii) $\implies$ (i): if $R_{\mm_j}$ is a DVR for some $j$, by part b) there is exactly one atom $p_j \in \mm_j$.  Since $R$ is atomic and $\pp_j$ is prime, every element of $\mm_j$ is divisible by $p_j$, so $\mm_j = (p_j)$ and $p_j$ is prime.
\end{proof}

\begin{lemma}
\label{1.3}
\label{02.3}
\label{CKLEMMA}
Let $(R,\mm)$ be a local atomic domain with residue field $k = R/\mm$.
\begin{enumerate}[noitemsep,topsep=0pt,label={\alph*}{\rm)},font=\rm]
\item Every element of $\mm \setminus \mm^2$ is an atom of $R$.
\item If two atoms $p,p' \in \mm \setminus \mm^2$ are associate, then $p \pmod{\mm^2}$ and $p' \pmod{\mm^2}$ generate the same one-dimensional $k$-subspace of $\mm/\mm^2$.
\item If $(R,\mm)$ is a primefree Cohen-Kaplansky domain then $k \cong \F_q$ is a finite field, $d = \dim_k \mm/\mm^2$ is finite, and $R$ has \emph{at least} $\# \PP^{d-1}(\F_q) = \frac{q^d-1}{q-1}$ atoms.
\item A primefree Cohen-Kaplansky domain has at least $3$ atoms.
\end{enumerate}
\end{lemma}
\begin{proof}
Parts a) and b) are left to the reader.  c) $R$ is a one-dimensional Noetherian local ring which is not a DVR, so
$1 < \dim_k \mm/\mm^2 < \aleph_0$.  By parts a) and b), choosing a nonzero element from each one-dimensional subspace of $\mm/\mm^2$ gives nonassociate atoms.  This set is in bijection with the set of lines through the origin of the $\F_q$-vector space $\mm/\mm^2$, hence with $\PP^{d-1}(\F_q)$.  d) By Theorem \ref{02.2} we reduce to the local case.  Then, with notation as above we have at least $\# \PP^{d-1}(\F_q) \geq \# \PP^1(\F_2) = 3$ atoms.
\end{proof}

\begin{thm}
\label{1.4}\mbox{ }
\begin{enumerate}[noitemsep,topsep=0pt,label={\alph*}{\rm)},font=\rm]
\item \cite[Thm. 13]{Cohen-Kaplansky46} For every prime power $q$ and $d \geq 2$, there is a primefree \CK$(\frac{q^d-1}{q-1},1)$-domain.
\item \cite[Cor., p. 475]{Cohen-Kaplansky46} If there is a primefree \CK$(n,1)$-domain for a prime number $n$, then $n$ is of the form $\frac{q^d-1}{q-1}$ for a prime power $q$ and an integer $d \geq 2$.
\item \cite[Cor. 7.2]{Anderson-Mott92} For any prime power $q$ and $d,e \in \Z^+$, $\F_q + t^e \F_{q^d}[[t]]$ is a \CK$(e \frac{q^d-1}{q-1} q^{d(e-1)},1;q)$-domain.  It is primefree unless $(d,e) = (1,1)$.
\end{enumerate}
\end{thm}

\begin{remark}
\label{1.5}
As mentioned above, Theorem \ref{1.4} implies there are primefree \CK$(n,1)$-domains for all $3 \leq n \leq 10$ and there is \emph{not} a primefree \CK$(11,1)$-domain.
\end{remark}
\noindent
We say a ring $R$ has \textbf{finite residue fields} if $R/\mm$ is finite for all $\mm \in \MaxSpec R$.  For a domain $R$ with fraction field $K$, and $I,J$ two $R$-submodules of $K$, we define \[ (I:J) = \{x \in K \mid x J \subset I \}. \]

\begin{thm}
\label{02.8}
Let $(R,\mm)$ be a local Cohen-Kaplansky domain with residue field $k = R/\mm \cong \F_q$, and put $d = \dim_k \mm/\mm^2$.  Let $\overline{R}$ be the normalization of $R$.
\begin{enumerate}[noitemsep,topsep=0pt,label={\alph*}{\rm)},font=\rm]
\item The ring $\overline{R}$ is a DVR, $\overline{R}$ is finitely generated as an $R$-module, and the ring $\overline{R}/(R:\overline{R})$ is finite. \\  Let $\overline{\mm}$ be the maximal ideal of $\overline{R}$ and put $\overline{k} = \overline{R}/\overline{\mm} \cong \F_{q^D}$. Let $\overline{r}\colon \overline{R} \ra \overline{k}$ be the quotient map.
\item The following are equivalent:
\begin{enumerate}[noitemsep,topsep=0pt,label=(\roman*),font=\rm]
\item The ideal $\mm^2$ is universal: every element of $\mm^2$ is divisible by every atom of $R$.
\item $R$ has $\frac{q^d-1}{q-1}$ atoms.
\item We have $R = \overline{r}^{-1}(\F_q)$.
\item We have $\overline{\mm} = \mm = (R:\overline{R})$.
\end{enumerate}
\item Under the equivalent conditions of part b), we have $d = D$.
\end{enumerate}
\end{thm}
\begin{proof}
a) The ring $\overline{R}$ is a DVR by the Krull-Akizuki Theorem.  By \cite[Thm. 2.4]{Anderson-Mott92}, $(R:\overline{R}) \supsetneq (0)$.  Thus
$\overline{R}/\overline{\mm}$ is finitely generated as a module over $R/\mm \cong \F_q$, hence is a finite field, say $\overline{R}/\overline{\mm} \cong \F_{q^D}$.  The \textbf{conductor} $(R:\overline{R})$ is the largest ideal of $R$ which is also an ideal of $\overline{R}$; since it is a nonzero ideal of a DVR with finite residue field, the ring $\overline{R}/(R:\overline{R})$ is finite. 
b) See
\cite[$\S$5]{Anderson-Mott92}. c) Since $\overline{R}$ is a DVR, we have $1 = \dim_{\F_{q^D}} \overline{\mm}/\overline{\mm}^2$, so $[\overline{\mm}:\overline{\mm}^2] = q^D$.
But $\mm = \overline{\mm}$, so \[[\mm:\mm^2] = [\overline{\mm}:\overline{\mm}^2] = q^D \] and \[ d = \dim_{R/\mm} \mm/\mm^2 = \dim_{\F_q} \mm/\mm^2 = \log_q [\mm:\mm^2] = D. \qedhere \]
\end{proof}
\noindent
Let $R$ be a ring, let $I$ be an ideal of $R$, let $q\colon R \ra R/I$ be the quotient map, and let $S$ be a subring of $R/I$.  Following Anderson-Mott \cite{Anderson-Mott92}, we call $q^{-1}(S)$ the \textbf{composite} of $R$ and $S$ over $I$.  Thus the condition
$R = \overline{r}^{-1}(\F_q)$ in Theorem \ref{02.8}b) above is that $R$ is the composite of $\overline{R}$ and $\F_q$
over $\overline{\mm}$.  On the other hand, Anderson-Mott characterize all Cohen-Kaplansky domains with finite residue fields (thus all primefree Cohen-Kaplansky domains) in terms of composites, as follows.

\begin{thm}\mbox{ }
\label{02.6}
\begin{enumerate}[noitemsep,topsep=0pt,label={\alph*}{\rm)},font=\rm]
\item Let $D$ be a semilocal PID with finite residue fields and maximal ideals $\MM_1,\ldots,\MM_m$.  Let $I = \MM_1^{e_1} \cdots \MM_m^{e_m}$ be an ideal of $D$, so $D/I \cong \prod_{j=1}^m D/\MM_j^{e_j}$.  For $1 \leq j \leq m$ let $S_j$ be a subring of $D/\mm_j^{e_j}$, and put $S = \prod_{j=1}^m S_j \subset D/I$.  Let $R$ be the composite of $D$ and $S$ over $I$. Then $R$ is a Cohen-Kaplansky domain with normalization $D$ and such that $(R:D) \supset I$.  Moreover $R$ has precisely $m$ maximal ideals, namely $\mm_j = \MM_j \cap R$ for $1 \leq j \leq m$.
\item Let $R$ be a Cohen-Kaplansky domain with finite residue fields, with normalization $\overline{R}$.  Then $\overline{R}$ is a semilocal PID with maximal ideals $\MM_1,\ldots,\MM_m$, and each $\overline{R}/\MM_j$ is finite. The conductor ideal $(R:\overline{R})$ is nonzero, so may be written as $\MM_1^{e_1} \cdots \MM_m^{e_m}$, and the subring $S = R/(R:\overline{R})$ of
$\overline{R}/(R:\overline{R}) \cong \prod_{j=1}^m \overline{R}/\MM_j^{e_j}$ may be decomposed as $\prod_{j=1}^m S_j$ with each $S_j$ a subring of $\overline{R}/\MM_j^{e_j}$.  $R$ is the composite of $\overline{R}$ and $S$ over $(R:\overline{R})$.
\end{enumerate}
\end{thm}
\begin{proof}
This is a rewording of \cite[Thm. 4.4]{Anderson-Mott92} suitable for our purposes.
\end{proof}

\begin{remark}
For a Cohen-Kaplansky domain $R$, its normalization $\overline{R}$ is a root extension: for all $r \in \overline{R}$ there is an $n \in \Z^+$ such that $r^n \in R$ \cite[Lemma 4.1]{Anderson-Mott92}.  Thus $S = R/(R:\overline{R}) \subset \overline{R}/(R:\overline{R})$ is also a root extension.  If $(R:\overline{R}) = \MM_1^{e_1} \cdots \MM_m^{e_m}$ then by the Chinese Remainder Theorem $\overline{R}/(R:\overline{R})$ decomposes as a product of $m$ finite local rings, with corresponding idempotents $\epsilon_1,\ldots,\epsilon_m$.  Since $\epsilon_j^n = \epsilon_j$ for all $n$, it follows that each $\epsilon_j$ lies in $S$.  The $S_j$ in Theorem \ref{02.6}b) is the projection $ S\epsilon_j$.
\end{remark}

\begin{cor}
\label{02.7}
Let $R$ be a Cohen-Kaplansky domain with finite residue fields which is the composite of $\overline{R}$ and $S = \prod_{j=1}^m S_j$ over $(R:\overline{R}) = \MM_1^{e_1} \cdots \MM_j^{e_j}$.  Then:
\begin{enumerate}[noitemsep,topsep=0pt,label={\alph*}{\rm)},font=\rm]
\item  $R$ has precisely $m$ maximal ideals,
$\mm_j = \MM_j \cap R$.
\item The localization $R_{\mm_j}$ is the composite of $\overline{R}_{\MM_j}$ and $S_j$ over $(R_{\mm_j}:\overline{R}_{\MM_j}) =\MM_j^{e_j}\overline{R}_{\MM_j}$.
     The completion
$\widehat{R_{\mm_j}}$ is the composite of $\widehat{\overline{R}_{\MM_j}}$ and $S_j$ over
$(\widehat{R_{\mm_j}}:\widehat{\overline{R}_{\MM_j}}) = \MM_j^{e_j} \widehat{\overline{R_{\MM_j}}} = \widehat{\MM_j}^{e_j}$.\footnote{Here we have used the canonical ring isomorphisms
     $\overline{R}/\MM_j^{e_j} \stackrel{\sim}{\ra} \overline{R}_{\MM_j}/\MM_j^{e_j} \overline{R}_{\MM_j} \stackrel{\sim}{\ra} \widehat{\overline{R}_{\MM_j}}/\MM_j^{e_j} \widehat{\overline{R_{\MM_j}}}$ to regard $S_j$ as a subring of the latter two rings.}
\item For $1 \leq j \leq m$, let $R_j$ be the composite of $\overline{R}$ and $S_j$ over $\MM_j^{e_j}$.  Then $R_j$ is a local Cohen-Kaplansky domain, $R = \bigcap_{j=1}^m R_j$ and $(R_j)_{\MM_j \cap R_j} = R_{\MM_j \cap R}$.
\end{enumerate}
\end{cor}
\begin{proof}
a) This is a property of Cohen-Kaplansky domains \cite[Thm. 2.4]{Anderson-Mott92}.  \\
b) Fix $1 \leq J \leq m$.  By Theorems \ref{02.2} and \ref{02.6}b),
the ring $R_{\mm_J}$ is a local Cohen-Kaplansky domain and is the composite of $\overline{R}_{\MM_J} = \overline{R_{\mm_J}} = \overline{R} \otimes_R R_{\mm_J}$ and
\[R_{\mm_J}/(R_{\mm_J}:\overline{R}_{\MM_J}) = R_{\mm_J}/(R_{\mm_J}:\overline{R} \otimes_R R_{\mm_J}) = R/(R:\overline{R}) \otimes_R R_{\mm_J} =  \bigg(\prod_{j=1}^m S_j\bigg) \otimes_R R_{\mm_J}\]  over $\MM_J^{e_J} \overline{R}_{\MM_J} = (R:\overline{R}) R_{\mm_J} = (R_{\mm_J}:\overline{R}_{\MM_J})$.  Consider the localization map
\[\varphi\colon \prod_{j=1}^m S_j \ra \bigg(\prod_{j=1}^m S_j\bigg) \otimes_R R_{\mm_J}, \]
which is the restriction of the localization map
\[\overline{\varphi}\colon \prod_{j=1}^m \overline{R}/\MM_j^{e_j} \ra
\bigg(\prod_{j=1}^m \overline{R}/\MM_j^{e_j}\bigg) \otimes_R R_{\mm_J}. \]
Because
\[
\bigg(\prod_{j=1}^m \overline{R}/\MM_j^{e_j}\bigg) \otimes_R R_{\mm_J}  = \prod_{j=1}^m \overline{R}_{\MM_j}/(\MM_j \overline{R}_{\MM_J})^{e_J} =
\overline{R}_{\MM_J}/(\MM_J \overline{R}_{\MM_J})^{e_J} \]
we find that $\overline{\varphi}$ factors through the projection $\prod_{j=1}^m \overline{R}/\MM_j^{e_j} \ra \overline{R}/\MM_J^{e_J}$ to
\[ \overline{\iota}\colon \overline{R}/\MM_j^{e_j} \stackrel{\sim}{\ra} \overline{R}_{\MM_J}/(\MM_J \overline{R}_{\MM_J})^{e_J}. \]
Thus $\varphi$ factors through the projection $\prod_{j=1}^m S_j \ra S_J$ to give an injection
\[ \iota\colon S_J \ra \bigg(\prod_{j=1}^m S_j\bigg) \otimes_R R_{\mm_J} = S_J \otimes_R R_{\mm_J}. \]
Thus $\iota$ is an injective localization map on the local Artinian $R$-algebra $S_J$.  Since every element of a local Artinian ring is either a nilpotent or a unit, any nonzero localization map on a local Artinian
ring is an isomorphism, so $\iota$ identifies $S_J$ with $(\prod_{j=1}^m S_j) \otimes_R R_{\mm_J}$.  The case of the completion is similar but easier. \\
c) The subring $\bigcap_{j=1}^m R_j$ is the set of elements $x \in \overline{R}$ such that for all $1 \leq j \leq m$ we have
$x \pmod {\MM_j^{e_j}} \in S_j$: manifestly, this is $R$.  By Theorem \ref{02.6}b), $R_j$ is a local Cohen-Kaplansky domain.  Finally, by part b) each of $(R_j)_{\MM_j \cap R_j}$ and $R_{\MM_j \cap R}$ is the composite of $\overline{R}_{\MM_j}$ and $S_j$ over
$\overline{R}/\MM_j^{e_j}$, so they are equal.
\end{proof}





\section{The Proofs}
\subsection{Proofs of Theorems \ref{MAINTHM} and \ref{PAULTHM}}

By Remark \ref{1.5} there are primefree $\CKn$-domains for $n \in [3,5]$, and by Theorem \ref{MAINCSCOR} if $p_1 < \ldots < p_m$ are primes there is a primefree CK$(\sum_{j=1}^m (p_j + 1),m)$-domain.  Thus to prove Theorem \ref{MAINTHM}, that primefree $\CKn$-domains exist for all $n \geq 3$ it is enough to prove Theorem \ref{PAULTHM}, that every $n \geq 6$ is of the form $\sum_{j=1}^m (p_j+1)$ for primes $p_1 < \ldots < p_m$ (Theorem \ref{PAULTHM}).

\textbf{Step 1:} We show that for all $n \geq 18$, there is a prime in the interval $(\frac{n}{2}-1,n-7]$.  The values $18 \leq n \leq 51$ can easily be checked by hand, so we may suppose $n \geq 52$.  We will make use of a sharpening of Bertrand's postulate due to Nagura \cite{Nagura52}: for all $x \geq 25$, there is a prime in the interval $(x,\frac{6x}{5}]$.  Applying this with $x = \frac{n}{2}-1$ we find that there is a prime number in the interval $(\frac{n}{2}-1,\frac{3n}{5}-\frac{6}{5}]$ hence also in the larger interval $(\frac{n}{2}-1,n-7]$.

\textbf{Step 2:} It suffices to show that for all $j \geq 0$, every $n \in [6,17 \cdot 2^j]$ is a sum of distinct $p_j+1$'s.  We show this by induction on $j$.  The base case $j = 0$ is an easy computation.  Suppose the result holds for some $j \geq 0$, and let
$n \in (17 \cdot 2^j,17 \cdot 2^{j+1}]$.  By Step 1, there is a prime number $p \in (\frac{n}{2}-1,n-7]$, so
\[ 6 \leq n - (p+1) < \frac{n}{2} \leq 17 \cdot 2^j. \]
By induction there are primes $p_1 < \ldots < p_m$ such that $n - (p+1) = \sum_{j=1}^m (p_j + 1)$, and thus
$n = \sum_{j=1}^m (p_j + 1) + (p+1)$.  Since $p+1 > \frac{n}{2}$, we have $p_j + 1 < \frac{n}{2}$ for all $j$, so $p > p_m$ and we have written $n$ as a sum of distinct $p_j+1$'s.

\begin{remark}
\label{03.1}
Theorem \ref{PAULTHM} is a variant of a result of H.-E. Richert \cite{Richert49}, who used Bertrand's postulate to show that every
$n \geq 7$ is the sum of distinct primes.
\end{remark}

\subsection{An Algebra Globalization Theorem}

\begin{thm}
\label{3.2} \label{03.2}
Let $q_1,\ldots,q_m$ be prime powers.
\begin{enumerate}[noitemsep,topsep=0pt,label={\alph*}{\rm)},font=\rm]
\item There is a number field $L$ and a sequence of distinct maximal ideals $\mm_1,\ldots,\mm_m$ of the ring of integers $\Z_L$ of $L$
such that for all $1 \leq j \leq m$ the quotient $\Z_L/\mm_j$ is a finite field of order $q_j$.
\item If there is a prime power $q$ and $b_1,\ldots,b_m$ such that $q_j = q^{b_j}$ for all $1 \leq j \leq m$, there is a finite degree field extension $L/\F_q(t)$ and a sequence of distinct maximal
ideals $\mm_1,\ldots,\mm_m$ of the integral closure $S$ of $\F_q[t]$ in $L$ such that for all $1 \leq j \leq m$ the quotient $S/\mm_j$ is a finite field of order $q_j$. \\
\end{enumerate}
\end{thm}

\begin{remark} Theorems of this kind appear in the literature. For instance, part (a) is a special case of \cite[Satz 1]{Hasse26}. However, we prefer to give a self-contained argument with the number field and function field cases treated on equal footing.
\end{remark}

\begin{proof}[Proof of Theorem \ref{3.2}]\mbox{ }

\textbf{Step 1:} Let $K$ be a field, and let $v_1,\ldots,v_{g+1}$ be inequivalent discrete valuations on $K$.  For $1 \leq i \leq g+1$ let $\hat{K}_i$ denote the completion of $K$ at $v_i$.  Let $d \in \Z^+$.  For $1 \leq i \leq g+1$, let $L_i$ be an \'etale $\hat{K}_i$-algebra -- i.e., a finite product of finite degree separable field extensions of $\hat{K}_i$ -- with $\dim_{\hat{K}_i} L_i = d$ and such that $L_{g+1}$ is a field.  Then by Krasner's Lemma and weak approximation, there is a separable field extension $L/K$ of degree $d$ and for all $1 \leq i \leq g+1$ a $\hat{K}_i$-algebra isomorphism $K \otimes_F \hat{K}_i \stackrel{\sim}{\ra} L_i$.  (We may write each $A_i$ as $\hat{K}_i[t]/(f_i(t))$ for a separable polynomial $f_i \in \hat{K}_i$.  Then we may take $L = K[t]/(f(t))$ where for all $i$, the coefficients of $f$ are sufficiently close to those of $f_i$ in the $v_i$-adic topology.  Thus we get a separable $K$-algebra $L$.  The condition that $L_{g+1}$ is a field ensures that $L$ is a field.)

\textbf{Step 2:} Recall that for all $d \in \Z^+$ there is a $p$-adic field with residue field $\F_{p^d}$: we may take the unique
degree $d$ unramified extension of $\Q_p$.  Let $g \in \Z^+$ be such that $q_1,\ldots,q_{j_1}$ are all powers of a prime $p_1$, $q_{j_1+1},\ldots,q_{j_2}$ are all powers of a prime $p_2$, $\ldots$, and so forth, up to $q_{j_{g-1}+1},\ldots,q_{j_g}$ -- note $j_g = m$ -- all powers of $p_g$.  Put $j_0 = 0$. Let $K = \Q$, for $1 \leq i \leq g$ let $v_i = \ord_{p_i}$, and let
\[L_i = \left(\prod_{j=j_{i-1}+1}^{j_i} L_j \right) \times M_i\] be a finite product of $p$-adic fields such that the residue cardinality of the valuation ring of $L_j$ is $q_k$ and $M_i$ is a ``fudge field'' chosen so that there is $D \in \Z^+$ such that
$\dim_{\Q_i} L_i = D$ for all $1 \leq i \leq g$.  The field extension $L/K$ obtained by the construction of Step 1 is the desired number field in part a).

\textbf{Step 3:} Suppose we are in the case considered in part b). For all $d \in \Z^+$, $\F_{q^d}((t))$ is a finite extension of $\F_q((t))$ with residue field $\F_{q^d}$.
We proceed as in Step 2 but with $K = \F_q(t)$, $g = 1$, $v_1 = \ord_t$ and $v_2 = \ord_{t-1}$.
\end{proof}

\subsection{Proof of Theorem \ref{SAURABHTHM}}
a) For $1 \leq j \leq M$ we are given a primefree CK$(n_j,m_j;q)$-domain $R_j$.  For $1 \leq k \leq m_j$ let $n_{jk}$ be the number of atoms in the $k$th maximal ideal of $R_j$ (under some ordering).  By Theorem \ref{02.2} there are primefree local CK$(n_{jk},1;q)$-domains $R_{j,1}, \ldots, R_{j,m_j}$ such that
$\sum_{j,k} n_{jk} = \sum_{j=1}^M n_j$.  So we may assume without loss of generality that $m_j = 1$ for all $j$.
\\ \forceindent
Suppose first that $q = 0$, and for each $1 \leq j \leq M$ we are given a primefree CK$(n_j,1;0)$-domain $A_j$.  By \cite[Thm. 9]{Cohen-Kaplansky46}, the completion $\hat{A_j}$ of the local ring $A_j$ is also a CK$(n_j,1;0)$-domain, so it is
no loss of generality to assume that each $A_j$ is complete.  Let $\overline{A_j}$ be the integral closure of $A_j$, a complete DVR, say with maximal ideal $(\pi_j)$.  Let $(A_j:\overline{A_j}) = (\pi_j^{e_j})$.  By Theorem \ref{02.6}, $A_j$ is the composite of $\overline{A_j}$ and $S_j = A_j/(A_j:\overline{A_j}) = A_j/(\pi_j^{e_j})$ over $(A_j:\overline{A_j}) = (\pi_j^{e_j})$.
\\ \forceindent
Since $\overline{A_j}$ is a complete DVR of characteristic $0$
with finite residue field $\F_{p_j^{a_j}}$, its fraction field $L_j$ is a finite extension of $\Q_{p_j}$.  Arguing as in the proof of
Theorem \ref{03.2} there is a number field $L$ and a set of finite places $v_1,\ldots,v_M$ of $L$ such that for all $j$, the completion of $L$ at $v_j$ is isomorphic to $L_j$.  Let $\MM_1, \ldots, \MM_M$ be the corresponding maximal ideals of the ring of integers $\Z_L$ of $L$, and let $D$ be the localization of $\Z_L$ at $\Z_L \setminus \bigcup_{j=1}^M \MM_j$, so $D$ is a semilocal PID
with precisely $M$ maximal ideals $\MM_1,\ldots,\MM_M$.  Moreover, for $1 \leq j \leq M$ we may identify the completion of $D$ with
respect to $\MM_j$ with $\overline{A_j}$, $\MM_j \overline{A_j}$ with $(\pi_j)$ and $D/\MM_j^{e_j}$ with $\overline{A_j}/(\pi_j^{e_j})$
and thus we may view $S_j$ as a subring of $D/\MM_j^{e_j}$.  Let $R_j$ be the composite of $D$ and $S_j$ over $\MM_j^{e_j}$ and
$R = \bigcap_{j=1}^M R_j$.  By Corollary \ref{02.7}, $R$ is a Cohen-Kaplansky domain with $M$ maximal ideals $\{\mm_j = \MM_j \cap R\}_{j=1}^M$ and we have $R_{\mm_j} = (R_j)_{\MM_j \cap R_j}$ and the completion of $R_{\mm_j}$ is $A_j$.  Thus $R$ is a primefree
CK$(\sum_{j=1}^M n_j,M;0)$-domain.
\\ \forceindent
Now suppose that $q \ne 0$, and for each $1 \leq j \leq M$ we are given a primefree CK$(n_j,1;q)$-domain $A_j$.  Then $\overline{A_j}$ is a complete DVR with finite residue field which is an $\F_q$-algebra, so its
fraction field $L_j$ is a finite extension of $\F_q((t))$.  We now argue as above except taking $L$ to be a finite extension of $\F_q(t)$ and $D$ to be the subring of $L$ consisting of functions regular at the places $v_1,\ldots,v_M$.

b) Let $q_1 = p_1^{a_1}, \ldots, q_m = p_m^{a_m}$ be prime powers and $d_1,\ldots,d_m \geq 2$.  For $1 \leq j \leq m$ let $K_j$
be the unramified extension of $\Q_{p_j}$ of degree $a_j d_j$, let $\overline{A_j}$ be its valuation ring, with maximal ideal $(\pi_j)$, let
$r_j\colon \overline{A_j}/(\pi_j) \stackrel{\sim}{\ra} \F_{q^{d_j}}$, and let $A_j = r_j^{-1}(\F_q)$.  By Theorem \ref{02.8} $A_j$ is a (necessarily primefree) CK$(\frac{q_j^{d_j}-1}{q_j-1},1;0)$-domain.  Applying part a), we get that there is a primefree
CK$(\sum_{j=1}^m \frac{q_j^{d_j}-1}{q_j-1},1;0)$-domain.
\\ \forceindent
Now suppose there is a prime power $q$ such that each $q_j$ is a power of $q = p^a$.  Then $p_1 = \ldots = p_m = p$ and $a_1,\ldots,a_m$ are
all divisible by $a$.  We run through the argument as above but with $K_j = \F_{q^{a_jd_j/a}}((t))$ for all $j$.

\subsection{Recalled Results and Conjectures in Additive Number Theory}

\begin{thm}[{Sylvester \cite{Sylvester84}}]
\label{SYLVESTERTHM}
Let $x,y$ be coprime positive integers.
Then:
\begin{enumerate}[noitemsep,topsep=0pt,label={\alph*}{\rm)},font=\rm]
\item The equation $ax + by = xy-x-y$ has no solution in non-negative integers $a,b$.
\item For all $N \geq xy-x-y + 1 = (x-1)(y-1)$, there are non-negative integers $a,b$ such that $ax + by = N$.
\end{enumerate}
\end{thm}

\begin{remark}
\label{SYLREMARK}
Let $x,y \in \Z^+$ with $\gcd(x,y) = d$.  Theorem \ref{SYLVESTERTHM} implies: if $N \geq
\frac{(a-d)(b-d)}{d}$ and $d \mid N$, then there are non-negative integers $a,b$ such that $ax + by = N$.
\end{remark}

\begin{thm}[{Helfgott \cite{Helfgott15}}]
\label{HELFTHM}
Every odd $n \geq 7$ is a sum of three primes.
\end{thm}

\begin{conj}[Goldbach]
\label{GOLDBACHCONJ}
Every even $n \geq 4$ is a sum of two primes.
\end{conj}

\begin{thm}[{Chudakov \cite{Chudakov37, Chudakov38}, Estermann \cite{Estermann38}, van der Corput \cite{vanderCorput37}}]
\label{CEVDCTHM}
The set of even integers which are sums of two primes has relative density $1$ in the set of even positive integers (and thus density $\frac{1}{2}$).
\end{thm}

\begin{conj}[Schinzel]
\label{SCHINZELCONJ} Let $f(x) = x^2+bx+c$, where $b$ and $c$ are integers of opposite parity and $3\nmid b$. There is a constant $N_0 = N_0(f)$ such that for all odd integers $n > N_0$ not belonging to $f(\Z)$, there are primes $p_1$ and $p_2$ with $n = f(p_1) + p_2$.
\end{conj}
\begin{remark} Conjecture \ref{SCHINZELCONJ} is a special case of a conjecture of Schinzel \cite{Schinzel63} generalizing the Goldbach conjecture. The conditions on $b, c$, and $n$ guarantee that the polynomial $n-f(x)$ is irreducible over $\Z$ and that $x(n-f(x))$ has no fixed divisor. Now the asymptotic prediction appearing as eq. (3) in \cite{Schinzel63} implies that the number of representations of $n$ in the form $f(p_1)+p_2$ tends to infinity as $n\to\infty$.
\end{remark}

\begin{thm}[van der Corput] \label{VDCTHM} If $f(x)$ satisfies the hypotheses of Conjecture \ref{SCHINZELCONJ}, then the set of odd $n$ representable in the form $f(p_1)+p_2$, with $p_1$ and $p_2$ prime, has relative density $1$ in the set of odd positive integers.
\end{thm}

\begin{proof} This is a special case of a more general theorem of van der Corput, announced in \cite{vanderCorput37} and proved in \cite{vanderCorput39}. See also \cite[Satz 2a]{Schwarz61}.
\end{proof}

\subsection{Proof of Theorem \ref{LITTLEMAIN}}
This follows from Theorem \ref{1.2} and Lemma \ref{CKLEMMA}d).

\subsection{Proof of Theorem \ref{MAINERTHM1}} We prove parts (a) and (b) simultaneously. Let $m,n \in \Z^+$.  We suppose: (i) $n \geq 3m$; (ii)  if $n$ is even then $m \geq 3$ and $n \geq 10$; (iii) if $n$ is odd then $m \geq 4$ and $n \geq 13$.

\textbf{Case 1:} Suppose $m \not \equiv n \pmod{2}$.  Then $n-3(m-3)$ is even and $n-3(m-3)\geq 9$, so in fact $n-3(m-3) \geq 10$.  By Theorem \ref{HELFTHM} there are primes $p_1,p_2,p_3$ such that
\[ n = (p_1 + 1) + (p_2 + 1) + (p_3 + 1) + \sum_{j=4}^m (2+1). \]

\textbf{Case 2:} Suppose $m \equiv n \pmod{2}$.  Then
$m \geq 4$.  Moreover $n -3(m-4) \geq 12$ and is even.  If $n-3(m-4) = 12$ then $n = 3m = \sum_{j=1}^m (2+1)$.  Otherwise $n-3(m-4) \geq 14$, so $n-3(m-4)-(3+1) \geq 10$ and is even, so by Theorem \ref{HELFTHM} there are primes $p_1,p_2,p_3$ such that
\[ n = (p_1 + 1) + (p_2 + 1) + (p_3 + 1) + (3+1) +
\sum_{j=5}^m (2+1). \]
In all cases Theorem \ref{SAURABHTHM} applies to show there is a primefree CK$(n,m;0)$-domain.

(c) By part (a), we may assume that $n$ is odd.  We appeal to a generalization due to Schwarz of Vinogradov's ``three primes theorem'' \cite[Hauptsatz, p. 25]{Schwarz60}. Schwarz's result implies that all sufficiently large odd $n$ are representable in the form $(p_1^2+p_1+1) + (p_2+1) + (p_3+1)$.  Apply Theorem \ref{SAURABHTHM}.


\subsection{Proof of Theorem \ref{THM1.7}}

a) If $n+2 = (p_1+1)+(p_2+1)$ for some primes $p_1$ and $p_2$, Theorem \ref{SAURABHTHM} applies to show there is a primefree CK$(n,2;0)$ domain. According to Theorem \ref{CEVDCTHM}, the set of even $n$ for which $n+2$ is so expressible has relatively density $1$ in the set of even positive integers. Similarly, if $n=(p_1^2 + p_1 + 1) + (p_2+1)$ for some primes $p_1$ and $p_2$, Theorem \ref{SAURABHTHM}a) applies to show there is a primefree CK$(n,2;0)$-domain. Theorem \ref{VDCTHM}, with $f(x)=x^2+x+2$, implies that the set of odd $n$ so representable has relative density $1$ in the set of odd numbers.

b) We argue as in part a) but apply Conjecture \ref{GOLDBACHCONJ} instead of Theorem \ref{CEVDCTHM}.

c) When $f(x)=x^2+x+2$, the condition $n \notin f(\Z)$ is satisfied by all odd integers $n$.  Now we argue as in part a) but apply Conjecture \ref{SCHINZELCONJ} instead of Theorem \ref{VDCTHM}.

\subsection{Proof of Theorem \ref{MAINERPAUL}}
Let $n$ be a prime number such there is a primefree CK$(n,1)$-domain.  By Theorem \ref{1.4}b), $n$ is of the form $\frac{q^d-1}{q-1}$.  By a result of Bateman-Stemmler \cite{Bateman-Stemmler62}, the number of primes $n \leq x$ of this form is at most $\frac{50 \sqrt{x}}{\log^2 x}$, for all sufficiently large $x$.  The result now follows from the Prime Number Theorem:
$\pi(x) = \# \{ \text{primes } p \leq x\} \sim \frac{x}{\log x}$.

\subsection{Proof of Theorem \ref{MAINQ}}

a) Let $R$ be a primefree atomic domain which is an $\F_q$-algebra.  We may assume $R$ is Cohen-Kaplansky, for otherwise it has infinitely many atoms.  Let $\mm$ be a maximal ideal of $R$.  By Theorem \ref{1.2}, $R$ has at least as many atoms as $R_{\mm}$, so we may assume that $(R,\mm)$ is local.  Then $R/\mm \cong \F_{q^a}$ and by Lemma \ref{1.3}c) $R$ has at least $\# \PP^{d-1}(\F_{q^a}) \geq \# \PP^1(\F_q) = q+1$ atoms.

b) Let $n \geq 8$.  Taking $a = 2+1$, $b = 4+1$ in Theorem \ref{SYLVESTERTHM}, we get that there are $x,y \in \N$ such that $x(2+1) + y(4+1) = n$.   By Theorem \ref{SAURABHTHM} there is a primefree CK$(n,x+y)$-domain of characteristic $2$.  The equation $x(2+1) + y(4+1) = n$ also has a solution in non-negative integers $x,y$ for $n \in \{3,5,6\}$, so there is a primefree $\CKn$-domain of characteristic $2$ for these values as well.  For $n \in \{4,7\}$ there is a primefree CK$(n,1)$-domain of characteristic $2$ by Theorem \ref{1.4}c).

c) By Theorem \ref{1.4} there are a primefree \CK$(n,1;q)$-domains for $n \in \{q+1,2q\}$, so by Theorem \ref{SAURABHTHM} for all $a,b \in \N$ there is a primefree \CK$(a(q+1)+b(2q),a+b;q)$-domain.  Because $q$ is a power of $2$, $q+1$ and $2q$ are coprime, so by
Theorem \ref{SYLVESTERTHM} every $n \geq (q+1-1)(2q-1) = 2q^2-q$ is of the form $a(q+1)+b(2q)$ for $a,b \in \N$.

d) Theorem \ref{1.4} gives primefree \CK$(n,1;q)$-domains for $n \in \{q+1,2q,q^2+q+1\}$, so by Theorem \ref{SAURABHTHM} for all
$a,b,c \in \N$ there is a primefree \CK$(a(q+1) + b(2q) + c(q^2+q+1),a+b+c;q)$-domain.  Since $q$ is odd we have $\gcd(q+1,2q) = 2$, so by Remark \ref{SYLREMARK} every even $n \geq (q-1)^2$ is of the form $a(q+1) + b(2q)$ for $a,b \in \N$.  If $n \geq 2q^2-q+1$ is odd, then $n \geq 2q^2-q+2$ and $n-(q^2+q+1) \geq (q-1)^2$ is even.

\subsection{Concerning Conjecture \ref{CONJ1}}
By Theorem \ref{LITTLEMAIN} if there is a \CK$(n,2)$-domain then $n \geq 6$ and if there is a \CK$(n,3)$-domain then $n \geq 9$.
\\ \forceindent
Using \texttt{PARI}/\texttt{GP}, we computed that there are 168 odd integers below $10^{10}$ \textbf{not} expressible in the form $(p_1^2+p_1+1)+(p_2+1)$, the largest being $1446379$. Using \texttt{Mathematica}, we checked that the 165 exceptional integers lying in $[7,10^{10}]$ nevertheless can be represented as a sum of two integers of the form $\frac{q^d-1}{q-1}$. Combined with Conjecture \ref{GOLDBACHCONJ}, we view this as evidence that all integers $n\ge 6$ admit a representation in that form. This would imply that there is a primefree \CK$(n,2;0)$-domain for all $n\ge 6$. If so, then as above for all $n \geq 6$ there is a primefree \CK$(n+3,3;0)$-domain, so the first part of Conjecture \ref{CONJ1} implies the second.
\\ \forceindent
Combining our computations with the verification of the Goldbach conjecture to $4\cdot 10^{18}$ \cite{OHP14}, it follows that both parts of Conjecture 1.8 hold for $n \le 10^{10}$.

\section{Final Comments}



\subsection{The connection with orders}
Coykendall-Spicer construct the domains used to prove Corollary \ref{MAINCSCOR} using local orders: for $S$ a finite nonempty set of primes, let $\Z^S$ be the localization of $\Z$ with respect to $\Z \setminus \bigcup_{p \in S} (p)$.  Let $K$ be any number field such that for all $p \in S$, there is a unique prime of $\Z_K$ lying over $p$.  Then the normalization $\overline{R}$ of $\Z^S$ in $K$ is a semilocal PID with $m$ maximal ideals.  Let $R$ be any $\Z^S$-order in $K$: i.e., a finitely generated $\Z^S$-subalgebra of $K$ with fraction field $K$.  Then $R$ is a Cohen-Kaplansky domain (cf. \cite[$\S$2]{Coykendall-Spicer12}).
\\ \forceindent
 One can replace $\Z^S$ with any semilocal PID $A$ with finite residue fields.  In the early stages of our work we took $A$ to be a semilocalization of the ring of integers of a number field $F$, $K/F$ a quadratic extension and worked with
relative quadratic orders $R$ in the normalization $\overline{R}$ of $A$ in $K$.  Then, using Theorem \ref{03.2} the residue fields of $A$ may be taken to be any multiset of finite fields, so the condition \emph{distinct primes}
may be replaced with \emph{arbitrary prime powers}.  When $R$ has squarefree conductor ideal $(R:\overline{R})$ then each
localization $R_{\mm}$ of $R$ is $\mm^2$-universal in the sense of Theorem \ref{02.6}b), and this is easy to prove using
the criterion $\mm = \overline{\mm} = (R_{\mm}:\overline{R_{\mm}})$.  This gives an alternate proof of Theorem \ref{SAURABHTHM}b) in the case $d_1 = \ldots = d_m = 2$.
\\ \forceindent
The proof of Theorem \ref{SAURABHTHM} shows that for every primefree \CK$(n;0)$-domain $R$ there is an order $\OO$ in a number field $K$ and a finite set $S$ of maximal
ideals of $\OO$ such that the group of divisibility of $R$ is isomorphic to the group of divisibility of $\OO^S =
(\OO \setminus \bigcup_{\mm \in S} \mm))^{-1} \OO$.  Here we have a \emph{localized order} rather than a $\Z^S$-order.
 But as in the proof of Theorem \ref{SAURABHTHM}, one sees that every group of divisibility of a local Cohen-Kaplansky domain in characteristic zero
arises as a $\Z_{(p)}$-order in a number field, and that in positive characterstic the same holds with $\Z_{(p)}$ replaced by the localization of $\F_p[t]$ at an irreducible polynomial.
It would be interesting
to compute the number of irreducibles in various local orders.

\subsection{A better Globalization Theorem?}
The main algebraic problem considered in this paper is the following:

\begin{ques}
\label{QUES3}
Let $R_1,\ldots,R_m$ be primefree Cohen-Kaplansky domains with $n_1,\ldots,n_m$ atoms.  Is there a primefree Cohen-Kaplansky
domain with $\sum_{j=1}^m n_j$ atoms?
\end{ques}
\noindent
Theorem \ref{SAURABHTHM} answers this question in the affirmative when $R_1,\ldots,R_m$ all
have the same characteristic.  Perhaps this is the only ``natural'' case: in what reasonable sense can domains of different
characteristics be combined?  But from the perspective of Question \ref{QUES2} it would be nice if one could combine
local building blocks of different characteristics.  Could it be that whenever there is a primefree CK$(n,1;q)$-domain
for $q > 0$ there is also a primefree CK$(n,1;0)$-domain?
\\ \forceindent
We end by discussing a possible strategy for showing this.  Every complete CK$(n,1)$-domain $R$ of positive characteristic is the composite of $\F_q[[t]]$ and a subring $S$ of
$\F_q[[t]]/(t^e)$ over $(t^e)$.  Every ring $\F_q[[t]]/(t^e)$ has an isomorphic copy $\OO_K/\MM^e$, where $K$ is a $p$-adic field with valuation ring $\OO_K$ and maximal ideal $\MM$ \cite{Necaev71}.  Thus we can build the ``corresponding characteristic $0$ CK-domain'' $\tilde{R}$: the composite of $\OO_K$ and the isomorphic copy of $S$ over $\MM^e$.  Does $\tilde{R}$ have the same number of irreducibles as $R$?
The group of divisibility of $R$ is isomorphic to $\Z \oplus \frac{ (\overline{R}/(R:\overline{R}))^{\times}}{(R/(R:\overline{R}))^{\times}            }$
\cite[Thm. 4.4]{Anderson-Mott92}.  But this is an isomorphism of abstract groups, whereas to count atoms we need the isomorphism to respect the partial orderings.  (The atoms are the minimal nonzero elements of the positive cone.)  Do these composites have isomorphic order structure?  If so, then Question \ref{QUES3} has an affirmative answer.

\end{document}